\documentclass[12pt]{amsart}
\usepackage[applemac]{inputenc}
\usepackage{
 ulem, 
 amsfonts}
\usepackage{amsmath,esint, amssymb, amsthm}

\usepackage[colorlinks=true, pdfstartview=FitV, linkcolor=blue, citecolor=blue, urlcolor=blue]{hyperref}
\usepackage{color}

\usepackage{fullpage}

\newcommand{\DD}{\textnormal{D}} 
\def\N{{\mathbb N}}
\date{}
\definecolor{sah}{rgb}{0.66,0.33, 0.04}
\definecolor{adel4}{cmyk}{1,0,0,0}
\definecolor{adel3}{rgb}{0.66,0.33, 0.04}
\definecolor{adel1}{cmyk}{0,0.20,1,0}
\definecolor{adel2}{cmyk}{0,0.40,1,0.30}
\definecolor{adel0}{rgb}{0.99,0.60, 0.30}
\definecolor{trut}{rgb}{0.99,0.80, 0.00}
\definecolor{trus}{rgb}{0.00, 0.50, 0.00}
 \definecolor{trust}{rgb}{0.99, 0.99, 0.80}
\definecolor{MaCouleur}{rgb}{0,0.9,0.3}
\def\virgp{\raise 2pt\hbox{,}}
\def\({\left(}
\def\){\right)}

\def\<{\langle}
\def\>{\rangle}

\theoremstyle{plain}

\newtheorem*{ex}{Example}

\newtheorem{Theo}{Theorem}
 \newtheorem{Lemm}{Lemma}

\newtheorem{Prop}{Proposition}
 \newtheorem{Coro}{Corollary}
 \newtheorem{Defin}{ Definition}
 
 \newtheorem{rema}{Remark}

\def\Xint#1{\mathchoice
   {\XXint\displaystyle\textstyle{#1}}%
   {\XXint\textstyle\scriptstyle{#1}}%
   {\XXint\scriptstyle\scriptscriptstyle{#1}}%
   {\XXint\scriptscriptstyle\scriptscriptstyle{#1}}%
   \!\int}
\def\XXint#1#2#3{{\setbox0=\hbox{$#1{#2#3}{\int}$}
     \vcenter{\hbox{$#2#3$}}\kern-.5\wd0}}

\def\av_#1{\Xint-_{#1}}

\newcommand{\ZZ}{\mathbb{Z}}  
  
\newcommand{\NN}{\mathbb{N}} 

\newcommand{\R}{{\mathbb R}}
\newcommand{\RR}{{\mathbb R}}

\newcommand{\lb}{{{\rm LBMO}}}
\newcommand{\baa}{\LMO}

\newcommand{\ba}{{{\it L^{\alpha}mo}_F}}
\newcommand{\bas}{{{\it L^{\alpha}mo}_{1+F}}}

\newcommand{\LL}{{|\ln r|^\alpha}}

\usepackage[textmath,displaymath,floats,graphics,delayed]{preview}
 \newcommand{\LMO}{ {\it L^{\alpha}mo}}

  \title[]{On the global well-posedness for Euler equations with unbounded vorticity}

\author[F. Bernicot]{Fr\'ed\'eric Bernicot}
\address{CNRS - Universit\'e de Nantes \\ Laboratoire de Math\'ematiques Jean Leray \\ 2, Rue de la Houssini\`ere F-44322 Nantes Cedex 03, France}
 \email{frederic.bernicot@univ-nantes.fr}

\author[T. Hmidi]{Taoufik Hmidi}
\address{IRMAR, Universit\'e de Rennes 1 \\ Campus de Beaulieu \\  35 042 Rennes cedex, France}
 \email{thmidi@univ-rennes1.fr}

\thanks{The two authors are partly supported by the ANR under the project AFoMEN no. 2011-JS01-001-01.}

\date{\today}


\begin{document}

\begin{abstract}
In this paper, we are interested in the global persistence regularity  for the 2D incompressible Euler equations in some function spaces allowing unbounded vorticities. More precisely,  we prove the  global propagation of the vorticity in some weighted Morrey-Campanato spaces and in this framework the velocity field is not necessarily  Lipschitz but belongs to the log-Lipschitz class $L^\alpha L,$ for some $\alpha\in (0,1).$
\end{abstract}

\maketitle

\begin{quote}
\footnotesize\tableofcontents
\end{quote}

\section{Introduction}
The motion of incompressible perfect flows evolving in the whole space is governed by the  Euler system described by the equations
\begin{equation}
\label{E}
 \left\{ 
\begin{array}{ll} 
\partial_t u+u\cdot\nabla u+\nabla P=0,\qquad x\in \mathbb R^d, t>0, \\
\textnormal{div }u=0,\\
u_{\mid t=0}= u_{0}.
\end{array} \right.    
     \end{equation}
 Here, the vector field   \mbox{$u: \mathbb R_+\times\mathbb R^d\to \RR^d$}  denotes  the velocity of the fluid particles and  the scalar function  
\mbox{$P$}   stands for  the   pressure.  It is a classical fact that the incompressibility condition leads to a closed system and the pressure can be recovered  from the velocity through  some singular operator. The literature on the   well-posedness theory for Euler system is  very abundant and a lot of results were obtained in various function spaces. For instance, it is well-known according to the work of Kato and Ponce \cite{Kato} that the system \eqref{E} admits a maximal unique solution in the framework of Sobolev spaces, namely $u_0\in W^{s,p},$ \mbox{with $s>\frac{d}{p}+1.$} This result was extended to H\"{o}lder spaces $\mathcal{C}^s, s>1$ by Chemin \cite{Ch1} and later by Chae \cite{Cha2} in the critical and sub-critical Besov spaces, see also \cite{Pak}. We point out that the common technical ingredient of these contributions   is the use of  the commutator theory but with slightly different difficulties.
Even though, the local theory for classical solutions is well-achieved, the global existence  of such solutions is still now an  outstanding open problem due to the  poor knowledge of  the conservation laws. However  this problem is affirmatively solved  for  some special cases like the dimension two and the axisymmetric flows without swirl. It is worthy pointing out that for these known cases the geometry of the initial data plays a central role through the special structure of their vorticities. Historically, we can fairly say that Helmholtz was the first to point out  in the seminal paper  \cite{Helm} the importance of the vorticity $\omega\triangleq \textrm{curl}\, u$ in the study of the incompressible inviscid  flows.  In  that paper he provided the foundations of the vortex motion theory by the establishment of some   basic  laws governing the  vorticity. Some decades later in the thirties of the last century, Wolibner proved in \cite{Wolib} the global existence of sufficiently smooth solutions in space dimension two. Very later in the mid of the eighties,   a rigorous connection between the vorticity and the global existence was performed by   Beale, Kato and Majda in \cite{Beale}. They proved the following blow up criterion: let $ u_0\in H^s,$ with $ s>\frac{d}{2}+1$ and denote by $T^\star$ the lifespan of the solution, then
$$
T^\star<+\infty\Longrightarrow \int_0^{T^\star}\|\omega(\tau)\|_{L^\infty}d\tau=+\infty.
$$ 
An immediate consequence of this criterion is the global existence of Kato's solutions in space dimension two. This follows from the conservation of the vorticity along the particle trajectories, namely the vorticity satisfies the Helmholtz equation
\begin{equation}\label{vorts1}
\partial_t\omega+u\cdot\nabla \omega=0.
\end{equation}
Recall that in this case the vorticity can be assimilated to the scalar $\omega=\partial_1 u^2-\partial_2 u^1$ and we derive from the  equation \eqref{vorts1}  an infinite family of conservation laws. For instance, for \mbox{every $p\in[1,\infty]$}
$$
\forall t\geq 0,\quad\|\omega(t)\|_{L^p}=\|\omega_0\|_{L^p}.
$$  
It seems that the standard methods used for the local theory cease to work in the limiting  space $H^{\frac{d}{2}+1}$ due to the lack of embedding in the  Lipschitz class. Nevertheless the well-posedness theory can be successfully implemented  in a slight modification of  this space in order to guarantee this embedding, take for example Besov spaces of type $B_{p,1}^{\frac{d}{p}+1}$,  for more details see for instance \cite{Cha2}. In this critical framework the BKM criterion cited before  is not known to work and should be replaced by the following one,
$$
T^\star<+\infty\Longrightarrow \int_0^{T^\star}\|\omega(\tau)\|_{B_{\infty,1}^0}d\tau=+\infty.
$$ 
In this class of initial data the global well-posedness in dimension two is not a trivial task and was proved by Vishik in \cite{Vishik2} through the use in an elegant way of the conservation of the Lebesgue measure by the flow. We mention that a simple proof of Vishik's result, which has the advantage to work in the viscous case, was given  in \cite{HK}. By using the formal $L^p$ conservation laws it seems that  we can go beyond the limitation fixed by the general theory of hyperbolic systems  and construct global weak solution for $p>1$ but for  the uniqueness we require in general the vorticity to be bounded. This was carefully done by Yudovich in his paper  \cite{Y1} following the tricky remark that the gradient of the velocity belongs to all $L^p$ with slow growth with respect to $p$:
$$
\sup_{p\geq2}\ \frac{\|\nabla v(t)\|_{L^p}}{p}<\infty.
$$
The uniqueness part is obtained by performing energy estimate and choosing  suitably the parameter $p$.
 In this new pattern   the velocity belongs to  the class of log-Lipschitz functions and this is sufficient to establish the existence and uniqueness of the flow map, see for instance \cite{Ch1}. 
The  real matter  at this level of regularity  concerns only  the uniqueness part which requires minimal regularity for the velocity  and the assumption of bounded vorticity is almost necessary in the scale of Lebesgue spaces. However  slight improvements have been carried out during the last decades by allowing the vorticity to be unbounded. For example, in \cite{Y2} Yudovich proved the uniqueness when the $L^p$-norms of the initial vorticity do not grow much faster than $\ln p:$
$$
\sup_{p\geq2}\ \frac{\|\omega_0\|_{L^p}}{\,\ln p}<\infty.
$$
We refer also to \cite{De,DM} for other extensions on the construction of global weak solutions. In \cite{Vishik1}, Vishik accomplished  significant studies for the existence and uniqueness problem with unbounded vorticities. He gave various results when the vorticity lies in the space $B_\Gamma\cap L^{p_0}\cap L^{p_1}$, where $p_0<2<p_1$  and $B_\Gamma$ is the borderline Besov spaces defined by
\begin{equation}\label{space-vis}
\sup_{n\geq 1}\ \frac{1}{\Gamma(n)}{\sum_{q=-1}^{n}\|\Delta_q\omega_0\|_{L^\infty}}<\infty.
\end{equation}
As an example, it was shown that for $\Gamma(n)=O(\ln n)$ there exits a unique local existence but the global existence is only proved when $\Gamma(n)=O(\ln^{\frac12} n).$ Nevertheless the propagation of the initial regularity is not well understood and Vishik were only able to prove that for the positive times the vorticity belongs  to the big class $B_{\Gamma_1}$ \mbox{with $\Gamma_1(n)=n\Gamma(n).$ } We point out that the persistence regularity for spaces which are  not embedded either in the Lipschitz class or in the spaces of conservation laws  is in general a difficult subject. Recently, in \cite{BK2}  the first author and Keraani were able to find a suitable space of initial data called log-$BMO$ space for which there is global existence and uniqueness without any loss of regularity. This space is strictly larger than the $L^\infty$ space and much smaller than the usual $BMO$ space.

\bigskip

The main goal of this paper is to continue this investigation and try to generalize the result of \cite{BK2}  to  a new collection of spaces which  are not comparable to the bounded class. To state our main result  we need to introduce the following spaces.
\begin{Defin}
Let $\alpha \geq0$ and $f:\R^2\to \R$ be a locally integrable function.
\begin{enumerate}
\item We say that $f$ belongs to the space $\baa$ if 
$$
\|f\|_{\baa}\triangleq\sup_{B\, \hbox{\tiny ball}\\
\atop 0<r\le \frac12}|\ln {r}|^\alpha \av_{B} \left|f-\av_{B}f \right|<\infty.
$$
\item Let $F:[1,+\infty[\to[0,+\infty[$. We say that $f$ belongs to $\ba$ if
$$
\|f\|_{\ba}\triangleq\|f\|_{\baa}+\sup_{B_1, B_2\\
\atop 2r_2\le r_1\le\frac12}\frac{|\av_{B_2}f-\av_{B_1}f|}{ F\left(\frac{|\ln(r_2)|}{|\ln(r_1)|}\right)}<\infty,
$$
where $ r_i$ denotes the radius of  the ball $B_i$, $|B|$ denotes the Lebesgue measure of the ball $B$ and the average $\displaystyle{\av_{B}f}$ is defined by $$
\av_{B}f\triangleq\frac1{|B|}\int_Bf(x)dx.
$$  
\end{enumerate}
\end{Defin}
For the sake of a clear presentation we  will first state a partial result and the general one will be given in Section \ref{sec:per}, Theorem \ref{apriori}.
\begin{Theo} 

\label{main} 
Take $F(x)=\ln x$ and assume that {$\omega_0\in L^p\cap \ba$}  with  $p\in ]1,2[$ \mbox{and $\alpha\in ]0,1[$. }
Then the $2d$ Euler equations admit  a unique  global solution 
$$\omega\in L^\infty_{loc}([0,+\infty[,L^p\cap\bas).
$$
\end{Theo}
Some remarks are in order.
\begin{rema} 
The regularity of the initial vorticity measured in the space $\baa$  is preserved globally in time. However we bring up a slight loss of regularity in the second part of the $\ba$ norm. Instead of $F$ we need $1+F.$ This appears as a technical artefact and we believe that we can remove it.
\end{rema}
\begin{rema} 
The case $\alpha=0$ is not included in our statement since it corresponds to the result of \cite{BK2}. However for $\alpha>1$ the vorticity must be bounded and  the velocity is Lipschitz and in this case the propagation in the space $\baa$ can be done without the use of the second part of the space $\ba$. The  limiting case $\alpha=1$ is omitted in our main result for the sake of simplicity but our computations can be performed as well with slight modifications especially when we deal with the regularity of the flow in Proposition \ref{prop}. 
\end{rema}

The proof of Theorem \ref{main} will be done in the spirit of the work of \cite{BK2}. We  establish a crucial logarithmic estimate for the composition in the space $\ba$ with a flow which preserves Lebesgue measure.  We prove in particular the key estimate 
$$
\|\omega(t)\|_{\baa}\le C\|\omega_0\|_{\ba}\left(1+V(t)  \right)\ln (2+V(t)), 
$$
where  $\displaystyle{ V(t)=\int_0^t\|u(\tau)\|_{L^{1-\alpha}L}d\tau}$ 
and the space $L^{1-\alpha}L$ is defined in Section \ref{log-lip}. We observe  from the preceding estimate that we can propagate globally in time the regularity in the space $\baa$ and the second part of the space $\ba$ is not  involved for the positive times.\\

The remainder of this  paper is organized as follows. In the  next  section we introduce some functional spaces and prove some of their basic properties. We shall also examine the regularity of the flow map associated to a vector field belonging to the class $L^\alpha L.$ In Section $3$ we shall establish a logarithmic estimate for a transport model  and we will see how to derive some of their consequences in the study of the inviscid flows. The proof of the main results will be given at the end of this section. We close this paper with an appendix covering the proof of some technical lemmata.

\section{Functional tools}
This section is devoted to some useful tools. We will firstly recall some classical spaces like Besov spaces and BMO spaces and give a short presentation of Littlewood-Paley operators. Secondly, we introduce the spaces $\baa$ and $\ba$ and discuss some of their important properties. We end  this section with  the study of log-Lipschitz spaces.\\
In the sequel we denote by $C$ any positive constant that may change from line to line and $C_{0}$ a real positive constant depending on the size of the initial data. We will use the following notations: for any non-negative real numbers $A$ and $B$, the notation  $A\lesssim B$ means that there exists a positive constant $C$ independent of $A$ and $B$ and such \mbox{that $A\leqslant CB$.}\\

\subsection{Littlewood-Paley operators}\label{subsection1}

To define Besov spaces we first introduce  the  dyadic partition of the unity, for more details see for  instance  \cite{Ch1}. There are  
two non-negative radial functions $\chi\in\mathcal{D}(\mathbb{R}^{2})$ and $\varphi\in\mathcal{D}(\mathbb{R}^{2}\backslash\{ 0\})$ such that 
$$\chi(\xi)+ \displaystyle \sum_{q\ge 0}\varphi(2^{-q}\xi)=1, \quad\forall \xi\in\mathbb{R}^{2},$$
$$\displaystyle\sum_{q\in\mathbb{Z}}\varphi(2^{-q}\xi)=1, \quad\forall \xi\in\mathbb{R}^{2}\backslash\{0\},$$
$$\vert p-q\vert\ge 2\Rightarrow\mbox{supp }{\varphi}(2^{-p}\cdot)\cap\mbox{supp }{\varphi}(2^{-q}\cdot)=\emptyset,$$
$$q\ge 1\Rightarrow \mbox{supp }{\chi}\cap\mbox{supp }{\varphi}(2^{-q}\cdot)=\emptyset.$$
Let $u\in\mathcal{S}^{\prime}(\RR^2)$, the  Littlewood-Paley operators are defined by
\begin{eqnarray*}
\Delta_{-1}u=\chi(\DD)u,\;\;\forall q\ge 0,\;\;\Delta_{q}u=\varphi(2^{-q}\DD)u\;\;\textnormal{and}\;\;S_{q}u=\displaystyle \sum_{-1\le p\le q-1}\Delta_{p}u.
\end{eqnarray*}
We can  easily check that in the distribution sense we have the identity $$u=\sum_{q\in \ZZ}\Delta_{q}u,\;\;\forall u \in \mathcal{S}^{\prime}(\RR^{2}).$$
Moreover, the Littlewood-Paley decomposition satisfies the property of almost orthogonality: for any $u,v\in\mathcal{S}^{\prime}(\RR^2),$
$$\Delta_{p}\Delta_{q}u=0\qquad \textnormal{if} \qquad \vert p-q \vert \geqslant 2 \qquad$$
$$\Delta_{p}(S_{q-1}u\Delta_{q}v)=0 \qquad \textnormal{if} \qquad  \vert p-q \vert \geqslant 5.$$\\
Let us note that the above operators $\Delta_{q}$ and $S_{q}$ map continuously $L^{p}$ into itself uniformly with respect to  $q$ and $p$. 
We also notice that these operators are of convolution type. For example for $q\in\ZZ,\,$ we have
$$\Delta_{-1}u=h\ast u,\quad \Delta_{q}u=2^{2q}g(2^q\cdot)\ast u,\quad\hbox{with}\quad g,h\in\mathcal{S},\quad \widehat{h}(\xi)=\chi(\xi),\quad \widehat{g}(\xi)=\varphi(\xi).
$$
Now we recall Bernstein inequalities, see for example \cite{Ch1}.
\begin{Lemm}\label{ber}
There exists a constant $C>0$ such that for all $q\in\NN\,,\,k \in \NN$ and for any tempered distribution $u$ we have  
\begin{eqnarray*}
\sup_{\vert\alpha\vert=k}\Vert\partial^{\alpha}S_{q}u\Vert_{L^{b}}\leqslant C^{k}2^{q\big(k+2\big(\frac{1}{a}-\frac{1}{b}\big)\big)}\Vert S_{q}u\Vert_{L^{a}}\quad \textnormal{for}\quad \; b\geqslant a\geqslant 1\\
C^{-k}2^{qk}\Vert{\Delta}_{q}u\Vert_{L^{a}}\leqslant \sup_{\vert\alpha\vert=k}\Vert\partial^{\alpha}{\Delta}_{q}u\Vert_{L^{a}}\leqslant C^{k}2^{qk}\Vert {\Delta}_{q}u\Vert_{L^{a}}.
\end{eqnarray*}
\end{Lemm}
 Using Littlewood-Paley operators, we can define Besov spaces as follows. 
For $(p,r)\in[1,+\infty]^2$ and $s\in\mathbb R,$ the    Besov 
\mbox{space $B_{p,r}^s$} is 
the set of tempered distributions $u$ such that
$$
\|u\|_{B_{p,r}^s}:=\Big( 2^{qs}
\|\Delta_q u\|_{L^{p}}\Big)_{\ell^{r}}<+\infty.
$$

We remark that the usual Sobolev space $H^s$ coincides with  $B_{2,2}^s$ for $s\in\RR$  and the H\"{o}lder space $C^s$ coincides with $B_{\infty,\infty}^s$ when $s$ is not an integer. 

The following embeddings are an easy consequence of  Bernstein inequalities, 
$$
B^s_{p_1,r_1}\hookrightarrow
B^{s+2({1\over p_2}-{1\over p_1})}_{p_2,r_2}, \qquad p_1\leq p_2\quad and \quad  r_1\leq r_2.
$$

Our next task is to introduce some new function spaces and to study some of their useful properties that will be frequently used along this paper.
\subsection{The $\LMO$ space}

Here the abbreviation  $Lmo$ stands for {\it logarithmic bounded mean oscillation}.

\begin{Defin} Let $\alpha \in[0,1]$ and $f:\R^2\to \R$ be a locally integrable function. We say that $f$ belongs to $\LMO$ if 
$$
\|f\|_{\LMO}:=\sup_{0<r\le \frac12}|\ln {r}|^\alpha \av_{B} \left|f-\av_{B}f \right|+\left(\sup_{|B|=1}\int_{B}|f(x)|dx \right)<\infty,
$$
where  the supremum is taken over all the balls $B$ of radius $r\leq \frac{1}{2}$. 
\end{Defin}

We observe that for $\alpha=0$ the space $\LMO$  reduces to the usual ${\it{Bmo}}$ space (the local version of ${\it{BMO}}$). It is also plain that the space $\LMO$ contains the class of continuous functions $f$  such that  
$$
\sup_{0<|x-y|\le\frac12}{|\ln|x-y||^\alpha\,{|f(x)-f(y)|}}<+\infty,
$$
that is the functions of modulus of continuity $\mu(r)=|\ln r|^{-\alpha}.$ There are  two elementary  properties that we wish to mention:
\begin{itemize}
 \item For $\alpha\in]0,1[$, consider  a ball  $B$ of radius $r$ and take  $k\geq 0$ with $2^k r\leq \frac{1}{2}$, then
\begin{eqnarray}
\nonumber \left| \av_{B} f - \av_{2^k B} f \right|& \lesssim& \|f\|_{\LMO} \sum_{\ell=0}^{k} (|\ln r|-\ell)^{-\alpha}\\
& \lesssim& \|f\|_{\LMO} |\ln r|^{1-\alpha }. \label{eq:ea} \end{eqnarray}
 \item For a ball $B$ of radius $1$ and $k\geq 1$, $2^k B$ can be covered by $2^{2k}$ balls of radius $1$, so
\begin{equation} \av_{2^k B} |f| \lesssim \|f\|_{\LMO}. \label{eq:ea2} \end{equation}
\end{itemize}

Next, we discuss some relations between the $\baa$ spaces and the frequency cut-offs.
\begin{Prop} \label{prop:delta} The following assertions hold true.
\begin{enumerate}
\item  Let $f\in \LMO,\,\alpha\in[0,1]$ and $n\in\NN^*$, then 
$$ \| \Delta_n f \|_{L^\infty} \lesssim n^{-\alpha} \|f\|_{\LMO},$$
and if $\alpha \in(0,1)$
\begin{equation*}\| S_nf \|_{L^\infty} \lesssim n^{1-\alpha} \|f\|_{\LMO}.
\end{equation*}
\item We denote by $\mathcal{R}_{ij}:=\partial_{x_i}\partial_{x_j}\Delta^{-1}$ the ``iterated'' Riesz transform. Then for every function $f\in \LMO \cap L^p$, with 
$p\in(1,\infty)$ and $\alpha\in(0,1)$, $$
\| S_n\mathcal{R}_{ij}f \|_{L^\infty} \lesssim n^{1-\alpha} \|f\|_{\LMO \cap L^p}.
$$
\end{enumerate}
\end{Prop}
This proposition yields easily to the following corollary.
\begin{Coro} We have the embedding $\baa\hookrightarrow B_\Gamma$, see the definition \eqref{space-vis}, with
\begin{itemize}
 \item $\Gamma(N)= \ln(N)$ if $\alpha=1$
 \item $\Gamma(N)=N^{1-\alpha}$ if $\alpha\in(0,1)$.
\end{itemize}
\end{Coro}

\begin{proof}[Proof of Proposition \ref{prop:delta}]
$\bf{(1)}$ 
The Littlewood-Paley operator $\Delta_n$ corresponds to a convolution  by $2^{2n}g(2^n\cdot)$ with $g$ a smooth function such that its Fourier transform is compactly supported away from zero. Therefore using the cancellation property of $g$, namely, $\int_{\RR^2}g(x)dx=0,$ we obtain
\begin{eqnarray*}\Delta_nf(x) &=& \int_{\R^2} 2^{2n} g( 2^{n}(x-y)) f(y) dy \\
&=& \int_{\R^2} 2^{2n} g( 2^{n}(x-y)) \left[f(y)-\av_{B(x,2^{-n})} f \right] dy,
\end{eqnarray*}
 Denote by $B\triangleq B(x,2^{-n})$ the ball of center $x$ and radius $2^{-n}$. Hence, due to the fast decay of $g$, it comes for every integer $M$
\begin{align*}
 \left| \Delta_nf(x) \right| & \lesssim  |B|^{-1} \sum_{k=0}^{n-1} 2^{-kM} \int_{2^k B} \left|f(y)-\av_{B} f \right| dy \\
 &+ |B|^{-1} \sum_{k=-1}^\infty 2^{-(n+k)M} \int_{2^{n+k} B} \left|f(y)-\av_{B} f \right| dy\\
 &\triangleq \hbox{I+II}. 
\end{align*}
To estimate the first sum $I$ we use the first inequality of (\ref{eq:ea}),
\begin{align*}
\hbox{ I }&\leq  \sum_{k=0}^{n-1} 2^{-k(M-2)} \av_{2^k B} \left|f-\av_{B} f \right|  \\
 & \lesssim \sum_{k=0}^{n-1} 2^{-k(M-2)} \av_{2^k B} \left|f-\av_{2^k B} f \right|  + \sum_{k=0}^{n-1} 2^{-k(M-2)} \left|\av_{B} f-\av_{2^k B} f \right|\\
 & \lesssim  \sum_{k=0}^{n-1} 2^{-k(M-2)} (|\ln(2^{k-n})|)^{-\alpha}\|f\|_{\LMO} + \sum_{k=0}^{n} 2^{-k(M-2)}\sum_{\ell=0}^k\frac{1}{|n-\ell|^\alpha} \|f\|_{\LMO} \\
 &\lesssim \sum_{k=0}^{n-1} 2^{-k(M-2)} \frac{1}{|k-n|^\alpha}\|f\|_{\LMO} + \sum_{k=0}^{n-1} 2^{-k(M-2)} \sum_{\ell=0}^k\frac{1}{|n-\ell|^\alpha} \|f\|_{\LMO}\\
 & \lesssim (1+n)^{-\alpha} \|f\|_{\LMO}. 
\end{align*}
As to the second sum we combine  (\ref{eq:ea}) and (\ref{eq:ea2})
\begin{eqnarray*}
\av_{2^{n+k} B} \left|f-\av_{B} f \right| &\le&\av_{2^{n+k} B} \left|f-\av_{2^nB} f \right|+\av_{2^{n} B} \left|f-\av_{B} f \right|\\
&\lesssim&  \|f\|_{\LMO}+n^{1-\alpha}\|f\|_{\LMO}\\
&\lesssim& \|f\|_{\LMO} n^{1-\alpha}.
\end{eqnarray*}
Consequently,
\begin{align*}
 \hbox{II} & \leq  \|f\|_{\LMO} \sum_{k\geq -1} 2^{-(n+k)(M-2)}n^{1-\alpha}\\
 & \lesssim n^{-\alpha} \|f\|_{\LMO}.
\end{align*}
The proof is now achieved  by combining these two estimates. 

Now let us focus on the estimate of  $S_n f$. We write according to the first estimate $(1)$ of the proposition
\begin{align*}
\|S_n f\|_{L^\infty}&\le \|\Delta_{-1}f\|_{L^\infty}+\sum_{q=0}^{n-1}\|\Delta_q f\|_{L^\infty}\\
&\lesssim  \|\Delta_{-1}f\|_{L^\infty}+\sum_{q=0}^{n-1}\frac{1}{ (1+q)^\alpha}\|f\|_{\LMO}\\
&\lesssim  \|\Delta_{-1}f\|_{L^\infty}+ n^{1-\alpha}\|f\|_{\LMO}.
\end{align*}
So it remains to estimate the low frequency part. For this purpose we  imitate the proof of $\|\Delta_n f\|_{L^\infty}$ with the following slight modification 
\begin{eqnarray*}
\Delta_{-1}(f)(x) &=& \int_{\R^2}  h( x-y) f(y) dy\\
& =& \int_{\R^2} h(x-y) \Big(f(y)-\av_{B(x,1)} f \Big)dy+ \av_{B(x,1)} f.
\end{eqnarray*}
Therefore we get
\begin{eqnarray*}
\|\Delta_{-1}f\|_{L^\infty}&\lesssim&\|f\|_{\LMO}+\sup_{x\in\mathbb{R}^2} \av_{B(x,1)} |f| \\
&\lesssim&\|f\|_{\LMO}.
\end{eqnarray*}

${\bf(2)}$ This can be  easily obtained by combining the first part of Proposition  \ref{prop:delta}  with the continuity on the $L^p$ space of  the localized Riesz transforms $\Delta_n \partial_i\partial_j\Delta^{-1}$  together with the help  of Bernstein inequality, for $n\geq 1$: 
\begin{eqnarray*}
\|S_n \mathcal{R}_{ij}f\|_{L^\infty}&\lesssim& \|\Delta_{-1}\mathcal{R}_{ij}f\|_{L^\infty}+\sum_{q=0}^{n-1}\|\Delta_q f\|_{L^\infty}\\
&\lesssim&  \|\mathcal{R}_{ij} f\|_{L^p}+\sum_{q=0}^{n-1}\frac{1}{ (1+q)^\alpha}\|f\|_{\LMO}\\
&\lesssim&  \|f\|_{L^p}+ n^{1-\alpha}\|f\|_{\LMO}.
\end{eqnarray*}
The proof of the desired result is now completed.
\end{proof}

Now we will introduce closed subspaces of the space  $\baa$ which play a crucial role in the study of Euler equations as we will see later  in the concerned section.

\subsection{The $\ba$ space}
It seems that the establishment  of the local well-posedness for Euler equations in the framework of $\baa$ spaces is quite difficult  and cannot be easily  reached by the usual methods. What we are able to do here is to construct the solutions in some weighted $\baa$ spaces  whose study will be  the subject of this section. 

Before stating the definition of these spaces we need the following concepts.
\begin{Defin}\label{def657}
Let $F:[1,+\infty[\to[0,+\infty[$ be a non-decreasing continuous function. 

$\bullet$ We say that $F$ belongs to the class $\mathcal{A}$ if there exists $C>0$ such that:
\begin{enumerate}
\item  Divergence at infinity: $\displaystyle{\lim_{x\to+\infty}F(x)=+\infty.}$
\item Slow growth: $\forall x,y\geq 1$
$$
F(x\,y)\le C\, (1+F(x))\,(1+F(y)).
$$
\item Lipschitz condition: $F$ is differentiable and 
$$
\sup_{x>1}|F^\prime(x)|\le C.
$$
\item Cancellation at $1$:
$$
\forall x\in[0, 1],\quad F(1+x)\leq C x.
$$
\end{enumerate}
$\bullet$ We say that $F$ belongs to the class $\mathcal{A}^\prime$ if it belongs to $\mathcal{A}$ and satisfies 
$$
\int_{2}^{+\infty}\frac{1}{x \,F(x)}dx=+\infty.
$$

\end{Defin}

\begin{rema}\label{rmq23}
\begin{enumerate}
\item From the slow growth assumption we see that necessarily the function $F$ should have at most a polynomial growth.
\item The assumption $(3)$ is only used through  Lemma $\ref{maj12}$ and could be  in fact relaxed for example  to $\|F^{(k)}\|_{L^\infty}<\infty$ for some $k\in \NN.$ But for the sake of simple presentation we limited our discussion to the case $k=1.$ 
\end{enumerate}
\end{rema}

\begin{ex}
\begin{enumerate}
\item For any $\beta\in ]0,1]$, the function $x\mapsto x^\beta -1$ belongs to the class $\mathcal{A} \setminus {\mathcal A}^\prime.$
\item For any $\beta\geq1$, the function $x\mapsto \ln^{\beta}(x)$ belongs to the class $\mathcal{A}$ and this function belongs to the class $\mathcal{A}^\prime$ only for $\beta=1.$
\item The function   $x\mapsto \ln x\,\ln\ln(e+x)$ belongs to the class $\mathcal{A}^\prime$.
\end{enumerate}
\end{ex}

We can now introduce the weighted $\baa$ spaces.

\begin{Defin}
Let $\alpha \in[0,1]$ and $F$ be in the class $\mathcal{A}$.  We define the space $\ba$ as the set of locally integrable functions $f:\RR^2\to\RR$ such that
$$
\|f\|_{\ba}\triangleq\|f\|_{\LMO}+\sup_{B_1, B_2}\frac{|\av_{B_2}f-\av_{B_1}f|}{ F\left(\frac{|\ln(r_2)|}{|\ln(r_1)|}\right)}<+\infty,
$$
where the supremum is taken over all the pairs of balls  \mbox{$B_2(x_2,r_2)$}   and  \mbox{$B_1(x_1,r_1)$}  in  \mbox{$\R^2$}   with  \mbox{$0<r_1\leq \frac12$}  and  \mbox{$2B_2\subset B_1$}.
Here, for a ball  \mbox{$B$}  and  \mbox{$\lambda>0$},  \mbox{$\lambda B$}  denotes the ball that is concentric with  \mbox{$B$}  and whose radius is  \mbox{$\lambda$}  times the radius of  \mbox{$B$}.
\end{Defin}

Now we list  some useful properties of these spaces that will be used later.
\begin{rema}\label{rmq67}
\begin{enumerate}
 \item The space $\lb$ introduced in \cite{BK2} corresponds to $\alpha=0$ and $F=\ln$. 
 \item Let $F_1, F_2\in \mathcal{A}$ such that $F_1\lesssim F_2$. Then we have the embedding
 $$
 \mathit{L^\alpha mo}_{F_1}\hookrightarrow  \mathit{L^\alpha mo}_{F_2}.
 $$
 \item For every  \mbox{$g\in \mathcal C^\infty_0(\R^2)$}  and  \mbox{$f\in \ba$}  one has
\begin{equation*}
\| g\ast f\|_{\ba}\leq \|g\|_{L^1}\|  f\|_{\ba}.
 \end{equation*}
Indeed, this property is just the consequence of Minkowski inequality and that the $\ba$-norm is invariant by translation.

 \end{enumerate}

 \end{rema}

The main goal of the following proposition is to discuss the link between the space of bounded functions and the space $\ba$. We will see in particular that under suitable assumptions on $F$ these spaces  are not comparable. More precisely we get the following.

\begin{Prop} 
\label{pro3}
Let $\alpha\in [0,1]$ and $f:\R^2\to\R$ be the radial function defined by
\begin{equation*}
 f(x)=\left\{ 
 \begin{array}{ll} \ln(1-\ln|x|) \qquad {\rm if}\quad |x|\leq 1\\
 0,\qquad \qquad {\rm if}\quad  |x|\geq 1.
 \end{array} \right.    
      \end{equation*}
The  following properties hold true.\\
\begin{enumerate}
\item The function $f$ belongs to $\baa$.
\item  For $F(x)= \ln x, x\geq 1,$ then $f\in \ba$.
\item For $\alpha\in ]0,1]$ and $F\in \mathcal{A}$ with $\ln\lesssim F$, the spaces $L^\infty$ and $\ba$ are not comparable.
\end{enumerate}
\end{Prop}
 \begin{proof}
 ${\bf{(1)}}$ There are at least two ways to get this result. The first one uses Spanne's criterion, see Theorem 2 of \cite{Spanne} and we omit here the details. However the second one is related  to  Poincar\'e inequality which states that for any ball for $B$ we have
$$ \av_{B} \left|f-\av_{B}f \right| \lesssim r\, \av_{B} |\nabla f|.
$$
For the example, it is obvious that  $|\nabla f(x)| \lesssim \frac{1}{|x|(1-\ln |x|)}\cdot$ So the quantity of the right-hand side in the Poincar\'e inequality is maximal for a ball $B$ centered at $0$ and consequently it comes
\begin{eqnarray*}
 r \,\av_{B} |\nabla f| &\lesssim &r^{-1} \,\int_0^{r} \frac{1}{1-\ln \eta} d\eta\\
 &\lesssim &\frac{1}{1-\ln r},
 \end{eqnarray*}
which concludes the proof of $f\in \baa$. \\
 ${\bf{(2)}}$ We reproduce the  arguments developed in \cite[Proposition 3]{BK2}, where it is proven that
\begin{equation} \left|\av_{B_2}f-\av_{B_1}f\right| \lesssim  \ln\left(\frac{1+|\ln(r_2)|}{1+|\ln(r_1)|}\right) + \mathcal{O}( |\ln(r_1)|^{-1}) + \mathcal{O}( |\ln(r_2)|^{-1}). \label{eq:end} \end{equation}
If $A\triangleq\frac{1+|\ln(r_2)|}{1+|\ln(r_1)|} \geq 2$ then $\mathcal{O}( |\ln(r_1)|^{-1}) + \mathcal{O}( |\ln(r_2)|^{-1})$ is bounded by $\ln(A)$ and so 
$$ \left|\av_{B_2}f-\av_{B_1}f\right| \lesssim \ln(A).$$
If $A\leq 2$, then $\ln(A)$ is equivalent to $A-1=\frac{\ln(r_1/r_2)}{1+|\ln(r_1)|} \gtrsim (1+|\ln(r_1)|)^{-1}.$ The latter inequality follows from the fact  $r_2\leq r_1/2$. Therefore we get 
\begin{eqnarray*}
 \mathcal{O}( |\ln(r_1)|^{-1}) + \mathcal{O}( |\ln(r_2)|^{-1}) \lesssim A-1 \approx \ln(A), 
 \end{eqnarray*}
which also gives 
$$ \left|\av_{B_2}f-\av_{B_1}f\right| \lesssim \ln(A).$$
Finally this ensures that $f\in \ba$, for every $\alpha \in[0,1]$.

 ${\bf{(3)}}$ According to Remark \ref{rmq67} we get the embedding $\mathit{L^\alpha mo}_{\ln}\hookrightarrow \ba$. Now by virtue of  the second claim of Proposition \ref{pro3} the function $f$ which is clearly  not bounded belongs to the \mbox{space $\ba$.} It remains to construct a function which is bounded but does not belong to the space $\ba.$
 Let $D_+$ be the upper half unit disc defined by 
 $$
 D_+\triangleq\big\{(x,y); x^2+y^2\le1, y\geq0  \big\}.
 $$
 By the same way we define the lower half unit disc $D_{-}.$ Let $r\le 1,$ denote by $B_r$  the disc of center zero and radius $r$ and let $g={\bf{1}}_{D_+}$ be  the characteristic function of $D_+.$ Easy computations yield
 \begin{equation*}
g(x)-\av_{B_{r}}g=
 \left\{ 
\begin{array}{ll} 
\frac12,\quad x\in B_r\cap D_+\\
-\frac12,\quad x\in B_r\cap D_{-}
\end{array} \right.    
     \end{equation*}
     Thus  we find for every $r\in(0,1)$
 \begin{eqnarray*}
 \av_{B_r}|g-\av_{B_r}g|&=&\frac12\cdot
 \end{eqnarray*}
 This shows that the function $g$ does not belong to $\baa$ for every $\alpha>0.$
 \end{proof}
Our next aim is to go over some refined  properties  of the weighted {\it{lmo}} spaces. One result that we will proved and  which seems to be surprising says that all the spaces $\ba$ are contained in the space $\it{L^\alpha mo}_{\ln}$. This  rigidity   follows from the cancellation property  of $F$ at the point $1$. More precisely, we shall show the following.
\begin{Prop} \label{prop:F} Let $\alpha \in(0,1]$ and $F\in {\mathcal A}$. Then $\ba \hookrightarrow \baa_{\ln}$.

Let $F:[1,+\infty)\to \RR_+$ defined by  $F(x)=\ln(1+\ln x)$ then $F\in{\mathcal A}$ and $\ba \subsetneq \baa_{\ln}$.
 \end{Prop}

\begin{proof}
Fix a function $f\in \ba$ and a point $x$ and set  $\phi(r) = \av_{B(x,r)} f$. From the definition of the space $\ba$ combined with the cancellation property of $F$ and its polynomial growth we get that for every $r\in(0,\frac{1}{2})$ and $k\geq 1$
\begin{align*}
 \left| \phi(r)-\phi(2^{-k} r) \right| & \leq \sum_{\ell = 0}^{k-1} \left| \phi(2^{-\ell} r)-\phi(2^{-\ell-1}r) \right| \\
 & \lesssim \sum_{\ell = 0}^{k-1} F\left(\frac{1+\ell + |\ln r|}{\ell +|\ln r|}\right) \\
 & \lesssim \sum_{\ell = 0}^{k-1} \frac{1}{\ell +|\ln r|}\\
 & \lesssim \ln\left(\frac{k + |\ln r|}{|\ln r|}\right).
\end{align*}
Then for $s<\frac{r}{2}$ choose $k\geq 1$ such that $2^{-k-1} r \leq s <2^{-k} r$ and so
$$ \left| \phi(r)-\phi(s) \right| \leq \left| \phi(r)-\phi(2^{-k} r) \right| + \left| \phi(s)-\phi(2^{-k} r) \right|.$$
As we have just seen, the first term is bounded by
$$ \ln\left(\frac{k + |\ln r|}{|\ln r|}\right) \approx \ln\left(\frac{1+|\ln s|}{|\ln r|}\right).$$
The second term is bounded as follows :
\begin{align*}
\left| \phi(s)-\phi(2^{-k} r) \right| & \leq \left| \phi(s)-\phi(2^{-k+1} r) \right| + \left| \phi(2^{-k-1}r)-\phi(2^{-k} r) \right| \\
 & \lesssim F \left(\frac{\ln s}{\ln(2^{-k+1}r)}\right) + F \left(\frac{1+k + |\ln(r)|}{k +|\ln(r)|}\right) \\
 & \lesssim \frac{1}{|\ln s|}\\
 & \lesssim \ln\left(\frac{|\ln s|}{|\ln r|}\right),
\end{align*}
where we used the cancellation property of $F$ and the fact that both $B(x,2s)$ and $B(x,2.2^{-k}r)$ are included into $B(x,2^{-k+1}r)$.
So combining these two previous estimates, it comes for every $x$, $r<\frac{1}{2}$ and $s\leq \frac{r}{2}$ 
\begin{equation} \left| \av_{B(x,r)} f - \av_{B(x,s)} f\right| \lesssim \ln\left(\frac{|\ln s|}{|\ln r|}\right). \label{eq:boule} \end{equation}
Now let \mbox{$B_2=B(x_2,r_2)$} and  \mbox{$B_1=B(x_1,r_1)$} two balls  with  \mbox{$0<r_1\leq \frac12$}  and  \mbox{$2B_2\subset B_1$}. We wish   to estimate $\left| \av_{B_2} f - \av_{B_1} f \right|$. First, it is clear  that  the interesting case  is  when at least the radius  $r_2$ is small, otherwise $r_1$ and $r_2$ are equivalent to $1$ and there is nothing to prove. So assume that $r_2\le  \frac{1}{100}$, then it is only sufficient to study the case where $r_1\leq \frac{1}{10}$. So let us only consider this situation : $r_2 \leq \frac{1}{100}$ and $r_1\leq \frac{1}{10}$. \\
Then we have
\begin{align*}
\left| \av_{B_2} f - \av_{B_1} f \right| & \lesssim  \left| \av_{B_2} f - \av_{\frac{r_1}{r_2} B_2} f \right| + \left| \av_{\frac{r_1}{r_2} B_2} f - \av_{B_1} f \right|.
\end{align*}
Applying (\ref{eq:boule}), the first term is bounded by  $\ln\left(\frac{\ln r_2}{\ln r_1} \right)$.
The second term can be easily bounded by $|\ln(r_1)|^{-1}$. Indeed, the two balls $\frac{r_1}{r_2} B_2$ and $B_1$ are comparable and of radius $r_1\leq \frac{1}{10}$. So there exists a ball $B$ of radius $r=5 r_1$, such that $\frac{2r_1}{r_2} B_2 \cup 2B_1 \subset B$.
Then we have
\begin{align*}
\left| \av_{\frac{r_1}{r_2} B_2} f - \av_{B_1} f \right| & \leq \left| \av_{\frac{r_1}{r_2} B_2} f - \av_{B} f \right| + \left| \av_{B} f - \av_{B_1} f \right| \\
 & \lesssim F\left(\frac{\ln r_1}{\ln r} \right)  \lesssim \frac{1}{|\ln r_1|}.  
\end{align*}
 Now, since $r_2\le\frac12 r_1$ then
\begin{eqnarray*}
\ln\left(\frac{\ln r_2}{\ln r_1}\right) & = & \ln\left(1+\frac{\ln(r_2/r_1)}{\ln r_1}\right)\\
&\geq&\ln\left(1+\frac{\ln 2}{|\ln r_1|}\right)\\
&\ge&C\frac{1}{|\ln r_1|}\cdot
\end{eqnarray*}
This concludes the proof of 
$$ \left| \av_{B_2} f - \av_{B_1} f \right| \lesssim \ln\left(\frac{1-\ln(r_2)}{1-\ln(r_1)}\right).$$
Hence we get  the inclusion $\ba \subset \baa_{\ln}$. \\
Then consider the specific function $F(\cdot)=\ln(1+\ln(|\cdot|))$. It is easy to check that $F\in{\mathcal A}$ and the function $f$ defined in Proposition \ref{pro3} belongs to $\baa_{\ln} \setminus \ba$. Indeed, (\ref{eq:end}) becomes an equality for this specific function $f$ with balls $B_1,B_2$ centered at $0$.
\end{proof}

The next proposition shows that for $\alpha=1$, the cancellation property of $F$ at $1$ (in $\ba$ space) can be ``forgotten'', since it is already encoded in the ${\it Lmo}$ space:

\begin{Prop} Let $\alpha =1$ and $F\in {\mathcal A}$. Then $\ba = {\it Lmo}_{1+F}$. Moreover, we have ${\it Lmo}= {\it Lmo}_{\ln}$.
\end{Prop}

\begin{proof} Since $F \leq 1+F$, it follows that $ \ba \subset {\it Lmo}_{1+F}$. 
Reciprocally, since for $t\geq 1$, $F(t) \simeq 1+F(t)$ (due to $F\in{\mathcal A}$), following the proof of Proposition \ref{prop:F} to prove that ${\it Lmo}_{1+F} \subset \ba$ it is sufficient to check that for every function $f\in {\it Lmo}$, every ball $B$ of radius $r<\frac{1}{4}$ then
\begin{equation} \left| \av_{B} f -\av_{2B} \right| \lesssim \frac{1}{|\ln(r)|}. \label{eq:to} \end{equation}
Indeed, the only difference between ${\it Lmo}_{1+F}$ and $\ba$ (where $F$ is replaced by $1+F$) is the loss of the cancellation property of $F$ at the point $1$ and this property was used in the previous proposition to check (\ref{eq:to}).\\
However, here since $\alpha=1$ (\ref{eq:to}) automatically holds since the function belongs to ${\it Lmo}$. Then producing the same reasoning as for Proposition \ref{prop:F}, we deduce that ${\it Lmo} \subset {\it Lmo}_{\ln}$, which yields ${\it Lmo}= {\it Lmo}_{\ln}$, since the other embedding is obvious.
\end{proof}

\subsection{Regularity of the flow map}\label{log-lip}
We shall continue in this section our excursion into function spaces by  introducing the log-Lipschitz class with exponent $\beta \in (0,1]$, denoted by $L^\beta L$ and  showing some links with the foregoing $\baa$ spaces. We next examine the regularity of the flow map associated to a vector field belonging to this \mbox{class $L^\beta L$. } We start with the following definition. We say that 
 a function $f$ belongs to the class $L^\beta L$ if 
$$ \|f\|_{L^\beta L}\triangleq\sup_{0<|x-y|<\frac12} \ \frac{|f(x)-f(y)|}{|x-y|\big|\ln|x-y|\big|^\beta}<\infty.  $$

Take now  a smooth divergence-free vector field $u=(u^1,u^2)$ on $\RR^2$ and $\omega=\partial_1 u^2-\partial_2 u^1$ its vorticity. It is apparent from straightforward computations that 
\begin{equation}\label{bs}
\Delta u=\nabla^\perp\omega.
\end{equation}
This identity leads through the use of the fundamental solution of the Laplacian to  the so-called Biot-Savart law.
 Now we shall  solve the \mbox{equation \eqref{bs}} when the source term belongs to the space $ \baa\cap L^p.$ Without going further into the details we restrict ourselves to the a priori estimates required for the resolution of this equation.
 
\begin{Prop} \label{coro} Let  $\alpha \in(0,1)$, $p\in(1,\infty)$ and   $\omega\in \LMO\cap L^p$  be the vorticity of the velocity $u$ given by the equation $(\ref{bs})$. Then $u\in L^{1-\alpha} L$ and there exists an absolute  constant $C>0$ such that
$$ \|u \|_{L^{1-\alpha} L} \leq C \|\omega\|_{\LMO\cap L^p}.$$
\end{Prop}
\begin{proof}
Let $N\in \N^\star$ be a given number that will be fixed later and $0<|x-y|<\frac12$. Using the mean value theorem combined with Bernstein inequality give
\begin{eqnarray*}
|u(x)-u(y)|&\le& |S_Nu(x)-S_Nu(y)|+2\sum_{q\geq N}\|\Delta_q u\|_{L^\infty}\\
&\lesssim& |x-y|\|\nabla S_N u\|_{L^\infty}+\sum_{q\geq N}2^{-q}\|\Delta_q \omega\|_{L^\infty}.
\end{eqnarray*}
From Proposition \ref{prop:delta}, it follows
\begin{eqnarray*}
|u(x)-u(y)|&\lesssim& N^{1-\alpha}\|\omega\|_{\LMO\cap L^p}|x-y|+\|\omega\|_{\LMO}\sum_{q\geq N}2^{-q}q^{-\alpha}\\
&\lesssim&\|\omega\|_{\LMO\cap L^p}\,N^{1-\alpha}\left(|x-y| +2^{-N}\right).
\end{eqnarray*}
By choosing $2^{-N}\approx |x-y|$ we find
$$
|u(x)-u(y)|\lesssim |x-y|\big|\ln|x-y|\big|^{1-\alpha}\,\|\omega\|_{\LMO\cap L^p}.
$$
This completes the proof of the proposition.
\end{proof}
We recall Osgood Lemma whose proof can be found for instance in \cite{bah-ch-dan}, page 128.
\begin{Lemm}[Osgood Lemma] \label{osgood1}
Let $a, A>0$, $\Gamma: [a,+\infty[\to\R_+$ be a non-decreasing function and $\gamma:[t_0, T]\to\R_+$ be a locally integrable function. Let $\rho:[t_0,T]\to [a,+\infty[$ be a measurable function such that
$$
\rho(t)\le A+\int_{t_0}^t \gamma(\tau) \Gamma(\rho(\tau))\,\rho(\tau) \, d\tau.
$$
Let $\displaystyle{\mathcal{M}(y)=\int_{a}^{y} \frac{1}{x\Gamma(x)}dx}$ and assume that $\displaystyle{\lim_{y\to+\infty}\mathcal{M}(y)=+\infty}$. Then
$$
\forall t\in[t_0, T],\quad \rho(t)\le \mathcal{M}^{-1}\Big(\mathcal{M}(A)+ \int_{0}^t\gamma(\tau)d\tau\Big).
$$
\end{Lemm}

In what follows we discuss the regularity of the flow map associated to a vector field belonging to the log-Lipschitz class. This precise description will be of great interest in the proof of the main result.
\begin{Prop}  \label{prop} Let  \mbox{$u$}  be a smooth divergence-free vector field belonging to $L^{1-\alpha} L,$ with $\alpha\in(0,1)$ and  \mbox{$\psi$}  be its  flow, that is the solution of the differential equation,
$$
\partial_t{\psi}(t,x)=u(t,\psi(t,x)),\qquad {\psi}(0,x)=x.
$$
Then, there exists $C\triangleq C(\alpha)>1$ such that for every  \mbox{$t\geq 0$}
$$
|x-y|< \ell(t)\Longrightarrow |\psi^{\pm1}(t,x)-\psi^{\pm1}(y)|\leq |x-y| e^{C V(t)|\ln|x-y||^{1-\alpha}},$$
where $\ell(t)\in(0,\frac12)$ is given by 
$$
  \ell(t) e^{C V(t)|\ln(\ell(t))|^{1-\alpha}}=\frac12\quad\hbox{and}\quad  V(t)\triangleq\int_0^t \|u(\tau)\|_{L^{1-\alpha} L}d\tau.
$$
Here we denote by $\psi^1$ the flow $\psi$ and $\psi^{-1}$ its inverse.
\end{Prop}

\begin{proof} It is well-known that for every  \mbox{$t\geq 0$}  the mapping  \mbox{$  x\mapsto \psi(t,x)$}  is a   Lebesgue  measure preserving homeomorphism  (see \cite{Ch1} for instance). We fix   \mbox{$x\neq y$}  such that $|x-y|<\frac12$ and  we define for $t\geq0$,
$$
z(t)\triangleq|\psi(t,x)-\psi(t,y)|.
$$
Clearly the function  \mbox{$z$}  is strictly positive and satisfies 
$$
z(t)\leq z(0)+C\int_0^t \|u(\tau)\|_{L^{1-\alpha} L}|\ln z(\tau)|^{1-\alpha} z(\tau)d\tau,
$$ 
as soon as $z(\tau)\leq \frac{1}{2}$, for all $\tau\in[0,t)$. Let $T>0$ and  $I\triangleq\big\{t\in [0,\ell(T)]\backslash \,\forall \tau\in[0,t], z(\tau)\leq \frac12\big\},$ where the value of $\ell(T)$ has been  defined in Proposition \ref{prop}.  We aim to show that the set $I$ is the full interval $[0,\ell(T)]$. First $I$ is a non-empty set since $0\in I$  and it  is an interval according to its definition. The continuity in time of the flow guarantees that $I$ is closed. It remains to show that $I$ is an open set of $[0,\ell(T)]$.  From the differential equation,
$$
\forall t\in I,\quad z(t)\leq z(0)+C\int_0^t \|u(\tau)\|_{L^{1-\alpha} L}(-\ln z(\tau))^{1-\alpha} z(\tau)d\tau.
$$
Accordingly, we infer
$$
-|\ln z(t)|^{\alpha}+|\ln z(0)|^{\alpha}\le C\alpha V(t),
$$
and this yields
$$
|\ln z(t)|\geq \big(|\ln z(0)|^{\alpha}-C\alpha V(t)\big)^{\frac1\alpha}.
$$
despite that 
\begin{equation}\label {c01}
C\alpha V(t)\le|\ln z(0)|^\alpha.
\end{equation}  Consequently
$$
z(t)\le e^{-\big(|\ln z(0)|^{\alpha}-C\alpha V(t)\big)^{\frac1\alpha}}.
$$
By virtue of  Taylor formula and since $\frac1\alpha-1>0$ we get
\begin{eqnarray*}
-\big(|\ln z(0)|^{\alpha}-C\alpha V(t)\big)^{\frac1\alpha}&=&-|\ln z(0)|+\frac1\alpha\int_0^{C\alpha V(t)}\big(|\ln z(0)|^{\alpha}-x\big)^{\frac1\alpha-1}dx\\
&\le&\ln z(0)+ CV(t)|\ln z(0)|^{1-\alpha}.
\end{eqnarray*}
It follows that
$$
z(t)\le z(0) e^{CV(t)|\ln z(0)|^{1-\alpha}}.
$$
Therefore to show that $I$ is open it suffices to make the assumption
$$
z(0) e^{CV(t)|\ln z(0)|^{1-\alpha}}<\frac12,
$$
which is satisfied when $z(0)< \ell(T).$ This last claim follows from the  increasing property of the function $x\mapsto x e^{CV(t)|\ln x|^{1-\alpha}}$  on the interval $[0, x_c]$ where $x_c<1$ is the unique real number satisfying $|\ln x_c|^\alpha=C(1-\alpha) V(t).$ From the definition of $\ell(t)$ we can easily check that $\ell(T)<x_c$ and \eqref{c01} is satisfied.

The proof of the assertion for $\psi^{-1}$ can be derived by performing  similar computations for the generalized flow defined by 
$$
\partial_t\psi(t,s,x)=u(t,\psi(t,s,x)),\quad\psi(s,s,x)=x
$$
and the flow $\psi^{-1}$ is nothing but $x\mapsto\psi(0,t,x).$

\end{proof}

\section{Regularity persistence} \label{sec:per}
The main object of this section is to examine the propagation of the initial regularity measured  in the spaces $\ba$ for the following transport model governed by a divergence-free vector field,
\begin{equation}
\label{T23}
 \left\{ 
\begin{array}{ll} 
\partial_t {{w}}+u\cdot\nabla w=0,\qquad x\in \mathbb R^2, t>0, \\
\textnormal{div }u=0,\\
w_{\mid t=0}=f.
\end{array} \right.    
     \end{equation}
Along the first part of this study we shall not prescribe any relationship between  the  solution $w$ and the vector field $u$. Once this study is achieved, we will apply this result for the inviscid  vorticity  where the vector field is induced by  the vorticity. This will enable us not only  to prove Theorem \ref{main} but also to state more general results on the local and global theory  extending the special case of $F(x)=\ln x$.
  \subsection{Composition in the space $\ba$ }
We begin with the following observation concerning the structure of the solutions to \eqref{T23}. Under reasonable assumptions on the regularity of  the velocity, the solution can be recovered from its initial data and the flow $\psi$ according to the formula
$ w(t)=f\circ\psi^{-1}(t)$. Thus  the study of the  propagation in the space $\baa$  reduces to 
the composition by a measure preserving map \mbox{in this space.}  We should note that this latter problem  can be easily solved as soon as the map is bi-Lipschitz (see \cite{BK} for composition in some BMO-type spaces by a bi-Lispchitz measure preserving map). 
In our context   the flow is not necessarily Lipschitz but in some sense very close to this class. It is apparent  according to Proposition \ref{prop} that $\psi$ belongs to the class $C^{s}$ for every $s<1.$ It turns out that working with a flow under the Lipschitz class has a profound effect and makes the  composition in the space $\baa$  very  hard to get. This is the principal reason why    we need to use the weighted subspace $\ba$ in order to compensate this weak regularity and  consequently to well-define the composition. Our result reads as follows.

\begin{Theo}  \label{decom}
Let  $\alpha\in(0,1),\, F\in \mathcal{A}$ and consider a smooth solution $w$ of the equation \eqref{T23} defined on $[0,T]$. Then there exists a constant $C\triangleq C(\alpha)>0$ such that the following holds true:
\begin{enumerate}
\item For every   $t\in [0,T]$
$$
\| w(t)\|_{\LMO}\leq C\|f\|_{\ba} \big(1+V(t)\big) F(2+V^{\frac1\alpha}(t)),
$$
with $V(t)\triangleq \displaystyle{\int_0^t\|u(\tau)\|_{L^{1-\alpha}L}d\tau}$.
\item  For every   $ t\in [0,T]$
$$
\|w(t)\|_{\LMO_{1+F}}\leq C\|f\|_{\ba}  F(2+V^{\frac1\alpha}(t)).
$$
\end{enumerate}
\end{Theo}
Before giving the proof, some remarks are in order.
\begin{rema}
\begin{enumerate}
\item  According to the first result of the foregoing theorem, the estimate of the solution in the space $\baa$ does not involve  the weighted part of the space $\ba,$ which is only required for the initial data. 
\item The estimate of the second part of Theorem $\ref{decom}$ is subjected to a slight loss. Indeed, instead of $F$ we put $1+F$. This is due to the fact that we need some cancellations for the difference of two averages and to avoid this loss more sophisticated analysis should be carried out and we believe   that this loss  is a technical artifact.
\end{enumerate}
\end{rema}

\begin{proof}[Proof of Theorem $\ref{decom}$] 
{\bf$(1)$} The proof will be done in the spirit of the recent work \cite{BK2}. First we observe that the solution is given by $w(t)=f\circ\psi^{-1}(t)$, where $\psi$ is the flow associated to the vector field $u.$ Therefore the estimate in the \mbox{space $\baa$} reduces to the stability by the right composition with a homeomorphism preserving Lebesgue measure with the prescribed regularity given in Proposition \ref{prop}. Since the \mbox{flow $\psi$} and its inverse share the same properties and the estimates that will be involved along the proof, we prefer for the sake of simple notation to use in the composition with $\psi$ instead of $\psi^{-1}.$   \mbox{Let $B=B(x_0,r)$} be the ball of center $x_0$ and radius $r\in(0,\frac12).$ We intend to give a suitable estimate for the quantity
$$
\mathcal{I}_r\triangleq |\ln r|^\alpha \fint_{B}|f\circ\psi- \fint_{B}(f\circ\psi)|dx.
$$

To reach this goal  we  use in  a crucial way the local regularity of the flow stated before in \mbox{Proposition \ref{prop}.} The estimate of $\mathcal{I}_r$ will require some discussions depending on a threshold value for $r$ denoted by $r_t$.  The identification of $r_t$ is related  to hidden arguments that will be clarified during  the proof. To begin with, fix a sufficiently  large  constant $\delta > \max(\sqrt{2}, 2C)$ (where $C$ is given by Proposition \ref{prop}) and  define $r_t$ as   the unique solution  in the interval $(0,\frac12)$ of the following equation
\begin{equation}\label{eqss:01} \delta  r_te^{ \delta V(t) |\ln(r_t)|^{1-\alpha}} =r_t^{ \frac12}.
\end{equation}
The existence and uniqueness can be easily proven by studying the variations of the  function
 \begin{equation}\label{hh1}h(r)\triangleq  \delta r^{\frac12}\,e^{ \delta V(t) |\ln(r)|^{1-\alpha}},\, r\in(0,1/2)
 \end{equation}
and using the fact that $h(\frac12)>1$ since $\delta>{\sqrt{2}}.$ We point out that $h$ is non-decreasing in the interval $[0,r_t]$ and $h(r)\le 1$ in this range. We have also the bound
\begin{equation}\label{m12}
C^\prime+C^\prime V^{\frac1\alpha}(t)\le |\ln r_t|\le C+C V^{\frac1\alpha}(t),\quad\hbox{for some}\quad C, C^\prime>0.
\end{equation}
Indeed, set $X=-\ln r_t$ then from \eqref{eqss:01} and  Young inequality
\begin{eqnarray*}
X&=&2\ln \delta +2\delta V(t) X^{1-\alpha}\\
&\le&C_1+C_1 V^{\frac1\alpha}(t)+\frac12 X.
\end{eqnarray*}
This gives the estimate of the right-hand side of \eqref{m12}. For the left one it is apparent from the equation on $X$ and its positivity that
\begin{eqnarray*}
X&\geq &2\ln \delta \quad\hbox{and}\quad X\geq 2\delta V(t) X^{1-\alpha}.
\end{eqnarray*}
Thus
$$
X\ge \ln \delta+ \frac{1}{2} (2\delta)^{\frac{1}{\alpha}} V^{\frac1\alpha}(t),
$$
which concludes the proof of (\ref{m12}). Before starting the computations for $\mathcal{I}_r$ we need to introduce the radius
\begin{equation}\label{rnot}
r_\psi\triangleq \delta\,r\, e^{\delta V(t) |\ln(r)|^{1-\alpha}}.
\end{equation}
Now we will check that for $r\in(0,r_t]$
\begin{equation}\label{eqss:1}
1\le\frac{|\ln r|}{|\ln r_\psi|}\le 2.
\end{equation}
The inequality of the left-hand side can be deduced  as follows. First it is obvious that $0<r<r_\psi$ and it remains to show that $r_\psi<\frac12$ whenever $r\in]0,r_t].$ For this purpose we show by  using simple arguments that   the function ${k}: r\mapsto r_\psi$ is non-decreasing in the \mbox{interval $(0,r_t].$} From this latter fact and \eqref{eqss:01} we find $r_\psi\le k(r_t)=r_t^{\frac12}\le \frac12$. Let us now move to  the second inequality of \eqref{eqss:1} and for this aim we start with studying the function
$$
g(x)=\frac{-x}{-x+a+b x^{1-\alpha}}, x\geq -\ln r_t ;\quad a\triangleq\ln \delta,\, b\triangleq\delta V(t). 
$$
We observe that the quotient $\frac{|\ln r|}{|\ln r_\psi|}$ coincides with $g(-\ln r)$. By easy computations we get
$$
g^\prime(x)=-\frac{a+b\alpha\,x^{1-\alpha}}{\big(-x+a+b x^{1-\alpha}\big)^2}<0.
$$
This yields in view of \eqref{eqss:01} and \eqref{m12}
\begin{eqnarray*}
g(x)&\le & g(-\ln r_t) \le \frac{\ln r_t}{\frac{\ln r_t}{ 2}} \le 2.
\end{eqnarray*}
\medskip
\noindent
The estimate of $\mathcal{I}_r$ depends whether the radius $r$ is smaller or larger than the critical \mbox{value  $r_t$.} 

{\centerline{\underline{\bf Case 1:} $0<r\le r_t .$ }}
\vspace{0,3cm}

As  \mbox{$\psi$}  is a  homeomorphism which preserves   Lebesgue  measure    then  \mbox{$\psi(B)$}  is an open connected 
set with  \mbox{$|\psi(B)|=|B|$}. Let us consider a Whitney covering of this open set $\psi(B)$, that consists in a collection of \mbox{balls  \mbox{$(O_j)_j$}} such that: 

\noindent \hspace{1cm} - The collection of double balls is a bounded covering:
$$
\psi(B)\subset \bigcup_j 2O_j.
$$

\noindent \hspace{1cm} - The collection is disjoint and for all  \mbox{$j$}, 
$$
O_j\subset \psi(B).
$$

\noindent \hspace{1cm} - The Whitney property is verified: the radius $r_j$ of $O_j$ satisfies
$$
r_{_j}\approx d(O_j, \psi(B)^c).
$$

We set  \mbox{$\tilde B\triangleq B(\psi(x_0), r_\psi)$} then according to Proposition \ref{prop} (and $\delta>2\rho(\alpha)$) we have  $\psi(B)\subset \tilde B$. It is easy to see from the invariance  of the Lebesgue measure by the flow that 
\begin{eqnarray*}
 \fint_{B}\Big|f\circ\psi- \fint_{B}(f\circ\psi)\Big|dx&=&\fint_{\psi(B)}\Big|f- \fint_{\psi(B)}f\Big|dx
\\
&\leq &    2  \fint_{\psi(B)}\Big|f- \fint_{\tilde B}f\Big|dx.
\end{eqnarray*}
Using the preceding notations 
\begin{eqnarray*}
\LL\fint_{\psi(B)}\Big|f- \fint_{\tilde B}f\Big|&\lesssim &\frac{\LL}{|B|}\sum_j |O_j| \, \fint_{2O_j}\Big|f- \fint_{\tilde B}f\Big|
\\
&\lesssim & \hbox{I}_1+\hbox{I}_2,
\end{eqnarray*}
with
\begin{eqnarray*}
\hbox{I}_1&\triangleq& \frac{\LL}{|B|}\sum_j |O_j| \, \fint_{2O_j}\Big|f- \fint_{2O_j} f\Big|\\
\hbox{I}_2&\triangleq& \frac{\LL}{|B|}\sum_j |O_j| \Big|\fint_{2O_j}f- \fint_{\tilde B} f\Big|.
\end{eqnarray*}
On one hand, since   \mbox{$\sum|O_j|\leq |B|$} and $r_j\le r<\frac12$ (due to $|O_j|\leq |\psi(B)|=|B|$)  then
\begin{eqnarray*}
\hbox{I}_1&\leq& \frac{1}{|B|}\sum_j |O_j|\frac{\LL}{|\ln(r_j)|^\alpha}\|f\|_{{\baa}}
\\
&\leq & \|f\|_{\baa} .	
\end{eqnarray*}
On the other hand, since $d(O_j, \tilde B) \leq r_\psi$ and $r_{\tilde B}=r_\psi$, it ensures that
$$ O_j \subset \frac{r_\psi}{r_j} O_j$$
and hence the two balls $Q_1\triangleq\frac{r_\psi}{r_j} O_j$ and $ \tilde{B}$ are comparable\footnote{Here we say that two balls $Q_1$ and $Q_2$ are comparable if $Q_1 \subset 4Q_2$ and $Q_2 \subset 4Q_1$.}.
This entails
$$ \Big|\fint_{Q_1} f - \av_{\tilde B}f  \Big| \lesssim \frac{1}{|\ln(r_\psi)|^\alpha} \|f\|_{\LMO}.$$
We point out that we have used the fact that for $0<r<r_t$ the radius $r_\psi$ of $\tilde B$ is smaller than $\frac12$.
Moreover according to the definition of the space $\ba$, it comes since $r_{Q_1}\leq \frac12$
\begin{eqnarray*}  \Big| \av_{ 2O_j} f  - \av_{Q_1} f  \Big| &\lesssim& \|f\|_{\ba}  F\left(\frac{\ln r_j}{\ln r_{Q_1}}\right) \\
 &\lesssim& \|f\|_{\ba}  F\left(\frac{\ln r_j}{ \ln r_{\psi}}\right).
\end{eqnarray*}
It follows that
\begin{eqnarray*} 
 \Big| \av_{ 2O_j} f  - \av_{\tilde B} f  \Big| &\lesssim& \|f\|_{\ba} \left( |\ln r_\psi|^{-\alpha}+ F\left(\frac{\ln r_j}{ \ln r_{\psi}}\right) \right).
\end{eqnarray*} 
Together with (\ref{eqss:1}) this estimate yields \begin{eqnarray*} 
\Big| \av_{ 2O_j} f  - \av_{\tilde B} f \Big| &\lesssim& \|f\|_{\ba} \left( |\ln r|^{-\alpha}+ F\left(\frac{\ln r_j}{ \ln r_{\psi}}\right) \right).
\end{eqnarray*}
Consequently,
\begin{eqnarray*}
\hbox{I}_2&\lesssim&  \|f\|_{\ba}+\|f\|_{\ba} \frac{|\ln r|^\alpha}{|B|} \left(\sum_j |O_j| F\left(\frac{\ln r_j}{\ln r_\psi}\right) \right).
\end{eqnarray*}   
For every  \mbox{$k\in\mathbb N$}  we set 
$$
u_k\triangleq\sum_{e^{-(k+1)}r< r_j\leq e^{-k}r} |O_j|,
$$
so that
\begin{eqnarray}
\label{eff}
\hbox{I}_2&\lesssim& \|f\|_{\ba} \left(1+ \frac{|\ln r|^\alpha}{|B|}\sum_{k\geq 0}u_k \ F\left(\frac{-1- k+\ln r}{\ln(r) +a+b|\ln r|^{1-\alpha}}\right)\right)\\
\nonumber&\triangleq& \|f\|_{\ba}+ \hbox{I}_3.
\end{eqnarray}
with
\begin{equation}\label{not11}
a\triangleq\ln \delta,\quad  b\triangleq\delta V(t). 
\end{equation}
The numbers $a$ and $b$  appeared before  in the definition of $r_\psi$ given in  \eqref{rnot}.
Let $N$ be  a real number that will be judiciously fixed later. We split the sum in the right-hand side of \eqref{eff} into two parts
$$
\hbox{I}_3=\sum_{k\leq N}(...)+\sum_{k> N}(.....)\triangleq\hbox{II}_{1}+\hbox{II}_{2}.
$$
Since  \mbox{$\sum u_k\leq |B|$} and $F$ is non-decreasing  then
\begin{eqnarray}
\label{ff}
\hbox{II}_{1}\lesssim  |\ln r|^\alpha\ F\left(\frac{1+N+|\ln r|}{|\ln(r)| -A-B|\ln r|^{1-\alpha}}\right).
\end{eqnarray}
 To estimate the term $\hbox{II}_2$ we need a refined bound for $u_k$  given below and  whose proof will be postponed to the end of this paper in \mbox{Lemma \ref{equivalence}} of the Appendix. 
\begin{equation}\label{precise}
u_k\lesssim  \delta r^2 e^{-k} e^{\delta (V(t)+1)(k-\ln(r))^{1-\alpha}}.
\end{equation}
By virtue of \eqref{precise} and  Lemma \ref{maj12} we get
\begin{eqnarray} \label{fff}
\nonumber \hbox{II}_{2}&\lesssim&  |\ln r|^\alpha  e^{-N} e^{b(N-\ln r)^{1-\alpha}}  F\left(\frac{1+N+|\ln r|}{|\ln r| -a-b|\ln r|^{1-\alpha}}\right)\\
&+& e^{-N} e^{b(N-\ln r)^{1-\alpha}}\frac{|\ln r|^\alpha}{|\ln r| -a-b|\ln r|^{1-\alpha}}\cdot 
\end{eqnarray}
So we choose $N=N(r)$ such that $e^{-N} e^{b(N-\ln r)^{1-\alpha}}=1$. Under this assumption we get
$$
\hbox{II}_{1}+\hbox{II}_{2}\lesssim  |\ln r|^\alpha\ F\left(\frac{1+N+|\ln r|}{|\ln r| -a-b|\ln r|^{1-\alpha}}\right)+\frac{|\ln r|^\alpha}{|\ln r| -a-b|\ln r|^{1-\alpha}}\cdot
$$ 
The condition on $N$ is also equivalent to
 $$N= 1+ b(N+|\ln r|)^{1-\alpha}.
 $$
Then from Young inequality
\begin{eqnarray}\label{ineqss}
N&\lesssim& 1+b^{\frac1\alpha}+b|\ln r|^{1-\alpha}
\end{eqnarray}
and therefore
\begin{align*}
 \frac{1+N+|\ln r|}{|\ln r| -a-b|\ln r|^{1-\alpha}} -1& =\frac{1+N+a+b|\ln r|^{1-\alpha}}{|\ln r_\psi|}\\
 &\lesssim \frac{1+b^{\frac1\alpha}+b|\ln r|^{1-\alpha}}{|\ln r_\psi|}\cdot
\end{align*}
Using the {\it cancellation property} of $F$ at the point $1$, that is $\displaystyle{\sup_{x\in(0,1)}\frac{F(1+x)}{x}<\infty}$, together with \eqref{eqss:1}, it comes
\begin{align*}
|\ln r|^\alpha F\left( \frac{1+N+|\ln r|}{|\ln r| -a-b|\ln r|^{1-\alpha}} \right) & \lesssim |\ln r|^\alpha \frac{1+b^{\frac1\alpha}+b|\ln r|^{1-\alpha}}{|\ln r_\psi|}  \\
& \lesssim \frac{|\ln r|^\alpha+|\ln r|^\alpha b^{\frac1\alpha}+b|\ln r|}{|\ln r|}  \\
&\lesssim1+b+b^{\frac1\alpha}|\ln r|^{\alpha-1}.
\end{align*} 
Since $r\in(0, r_t]$ and according to \eqref{m12} we find 
$$1+b+b^{\frac1\alpha}|\ln r|^{\alpha-1} \lesssim 1+V(t)
$$and so
\begin{align*}
|\ln r|^\alpha F\left( \frac{1+N+|\ln r|}{|\ln r| -a-b|\ln r|^{1-\alpha}} \right) & \lesssim \left(1+V(t)\right).
\end{align*} 
It follows that
$$ \hbox{II}_1+\hbox{II}_2 \lesssim  1+ V(t)+\frac{|\ln r|^\alpha}{|\ln r| -a-b|\ln r|^{1-\alpha}}.$$
To estimate the last term we use \eqref{eqss:1} 
\begin{eqnarray*}
\frac{|\ln r|^\alpha}{|\ln r| -a-b|\ln r|^{1-\alpha}}&=&\frac{|\ln r|^\alpha}{|\ln r_\psi|}\\
&\leq&{|\ln r|^{\alpha-1}}\\
&\lesssim& 1.\end{eqnarray*}
 
 Finally, we get
 \begin{equation}\label{basse3}
 \sup_{0< r\le r_t}\left(\LL\fint_{\psi(B)}\Big|f- \fint_{\tilde B}f\Big|+\mathcal{I}_r\right)\lesssim \|f\|_{\ba} \left(1+V(t)\right).
 \end{equation}

\medskip
Let us now move to the second case.

\centerline{\underline{{\bf Case 2:}} $r_t\le r\leq \frac12$.}

\vspace{0,3cm}

According to  \eqref{m12}  
\begin{eqnarray}\label{tita1}
\nonumber |\ln r|&\le& |\ln r_t|\\
&\lesssim& 1+V^{\frac1\alpha}(t)
\end{eqnarray}
which yields in turn
\begin{equation}\label{basse1}
\LL\fint_{B}\Big|f\circ\psi- \fint_{B}f\circ\psi\Big|\lesssim \big(1+V(t)\big)\fint_{\psi(B)}|f|.
\end{equation}
Let $\tilde{O}_j $ denote the ball which is concentric  to $O_j$ and whose radius is equal to $1/2$. We can write by the definitions,
\begin{eqnarray*}
\fint_{\psi(B)}|f|&\le &\frac{1}{|B|}\sum_{j}|O_j|\fint_{2 O_j}\Big|f-\fint_{\tilde{O}_j }f\Big| +\frac{1}{|B|}\sum_{j\in\mathbb{N}}|O_j|\fint_{\tilde{O}_j} |f| \\
&\le& \frac{1}{|B|}\sum_{j}|O_j|\fint_{2 O_j}\Big|f-\fint_{\tilde{O}_j }f\Big| + \sup_{|B|=1}\fint_{B}|f| \\
&\le&\|f\|_{\ba} \frac{1}{|B|}\sum_{j}|O_j|F(-\ln {r_j})+\|f\|_{\baa}.
\end{eqnarray*}
Now 
reproducing  the same  computations as for the first case leads to 
\begin{eqnarray*}
 \frac{1}{|B|}\sum_{j}|O_j| F(-\ln {r_j})&\le&
\frac{1}{|B|}\sum_{k\in\mathbb{N}}u_k F(k-\ln {r})\\
&\lesssim& F(N-\ln{r})\Big(1+ e^{-N} e^{b(N-\ln(r))^{1-\alpha}}\Big).
\end{eqnarray*}
This computation still holds as soon as $e^{-N} r \leq \ell(t)$, since we use Proposition \ref{prop} for the \mbox{scales $e^{-k}r$} with $ k\geq N$.
We choose $N$ such that $  e^{-N} e^{b(N-\ln(r))^{1-\alpha}}= 1$ and we check that this choice legitimates the previous estimate since $e^{-N}r \leq \ell(t)$.
Consequently, we get  from (\ref{ineqss}) and \eqref{tita1} 
$$ N+|\ln r| \lesssim 1+V^{\frac1\alpha}(t).$$
We then obtain
\begin{equation}\label{basse113}
\fint_{\psi(B)}|f|\lesssim F\big(2+V^{\frac1\alpha}(t)\big)\, \|f\|_{\ba}
\end{equation}
and 
\begin{eqnarray}\label{basse13}
\nonumber\mathcal{I}_r&\lesssim&\LL\fint_{\psi(B)}\Big|f\circ\psi- \fint_{\tilde B}f\Big|\\
&\lesssim&  \big(1+V(t)\big)F\big(2+V^{\frac1\alpha}(t)\big)\, \|f\|_{\ba}.
\end{eqnarray}
Finally, we have obtained for $r\in [r_t,1/2]$
$$ \mathcal{I}_r \lesssim \big(1+V(t)\big) F\big(2+V^{\frac1\alpha}(t)\big)\, \|f\|_{\ba}.$$

Putting together the estimates of the case $1$ and the case $2$ yields
\begin{equation}\label{basse213}
\|f\circ \psi\|_{\LMO}\lesssim \big(1+V(t)\big)F\big(2+V^{\frac1\alpha}(t)\big)\, \|f\|_{\ba}.
\end{equation}

{$\bf(2)$}  To deal with the second term in the  \mbox{$\ba$}-norm we will make use of  the arguments   developed  above for $\baa$ part.
Take   \mbox{$B_2=B(x_2,r_2)$}   and  \mbox{$B_1=B(x_1,r_1)$}  two balls  \mbox{with   \mbox{$r_1\leq 1$}}  and  \mbox{$2B_2\subset B_1$} and let us see how to estimate  the quantity
$$
\mathcal{J}\triangleq\frac{|\av_{B_2}f\circ\psi-\av_{B_1}f\circ\psi|}{ 1+F(\frac{\ln r_2}{\ln r_1})}\cdot
$$
There are different  cases to consider.

\vspace{0,3cm}

{\centerline{\underline{\bf Case 1:}  $r_t\le r_1\le \frac12$.}}

Using \eqref{basse113} 
\begin{eqnarray*}
\frac{|\av_{B_1}f\circ\psi|}{1+ F(\frac{\ln r_2}{\ln r_1})}&\le&\av_{\psi(B_1)}|f|    \\
&\lesssim&F\big(2+V^{\frac1\alpha}(t)\big)\|f\|_{\ba} .
\end{eqnarray*}
If $r_2> r_t$ then by repeating the same arguments for the quantity involving $B_2$, it comes 
$$\mathcal{J} \lesssim F\big(2+V^{\frac1\alpha}(t)\big)\|f\|_{\ba}.$$
If $r_2\leq r_t$ then we estimate the average on $\psi(B_2)$  by using \eqref{basse3} and \eqref{m12}
\begin{align*}
 \av_{\psi(B_2)}|f| & \leq \av_{\psi(B_2)}\Big|f-\av_{\tilde{B}_2} f\Big| + \Big|\av_{\tilde{B}_2} f \Big| \\
 &\lesssim |\ln r_2|^{-\alpha}(1+V(t))\|f\|_{\ba}+ \Big|\av_{\tilde{B}_2} f \Big| \\
 &\lesssim \|f\|_{\ba}+ \Big|\av_{\tilde{B}_2} f \Big|,
\end{align*}
where \mbox{$\tilde B_i\triangleq  B(\psi(x_i), r_{i,\psi}), i=1,2$}  and $r_{i,\psi}$ is the radius associated to $r_i,$ which was introduced in \eqref{rnot}. It remains to treat the last term of the above inequality. For this goal we write
\begin{align*}
 \Big|\av_{\tilde{B}_2} f \Big|
 & \lesssim  \Big|\av_{\tilde{B}_2} f - \av_{{B}_{(\psi(x_2),1/2)}} f\Big|+\sup_{B, r=\frac12}\Big| \av_{B}f \Big| \\
 & \lesssim F(|\ln r_{2,\psi}|) \|f\|_{\ba}+\|f\|_{\baa}.
\end{align*}
This yields in view of \eqref{eqss:1} and the Definition \ref{def657}
\begin{align*}
 \frac{|\av_{B_2}f\circ\psi|}{1+ F(\frac{\ln r_2}{\ln r_1})}
 & \lesssim \Big(1+\frac{F(|\ln r_{2,\psi}|)}{1+F(\frac{\ln r_2}{\ln r_1})}\Big) \|f\|_{\ba}\\
 &\lesssim  \Big(1+\frac{F(|\ln r_2|)}{1+F(\frac{\ln r_2}{\ln r_1})}\Big) \|f\|_{\ba}\\
 &\lesssim \big( 1+F(\ln r_1)\big)\|f\|_{\ba}.
\end{align*}
Since $r_1\in (r_t, \frac12)$ then using \eqref{m12} we find
\begin{align*}
 \frac{|\av_{B_2}f\circ\psi|}{1+ F(\frac{\ln r_2}{\ln r_1})}
 &\lesssim F(2+V^{\frac1\alpha}(t))\|f\|_{\ba}.
\end{align*}
Finally we get for $r_2\le r_t$
$$
\mathcal{J} \lesssim F\big(2+V^{\frac1\alpha}(t)\big)\|f\|_{\ba}.$$
\vspace{0,3cm}
To achieve   the proof of the second part of Theorem \ref{decom}, it remains to analyze the last case:  
\centerline{\underline{\bf Case 2:} $0<r_1\le r_t$.}
\vspace{0,1cm}

 We decompose  $\mathcal{J}$ as follows:
$$
\mathcal{J}= \frac{\mathcal{J}_{1}+\mathcal{J}_{2}+\mathcal{J}_3}{{1+F(\frac{\ln r_2}{\ln r_1})}},
$$
with 
\begin{eqnarray*}
\mathcal{J}_{1}&\triangleq&  |\av_{\psi(B_2)}f-\av_{\tilde B_2}f|+ |\av_{\psi(B_1)}f-\av_{\tilde B_1}f| \\
\mathcal{J}_{2}&\triangleq&{|\av_{\tilde B_2}f-\av_{2\tilde B_1}f|}
 \\
\mathcal{J}_{3}&\triangleq&|\av_{\tilde B_1}f-\av_{2\tilde B_1}f|.
\end{eqnarray*}
The first term $J_1$ can be handled as for \eqref{basse3} and we get by \eqref{m12}
\begin{eqnarray*}
\Big|\av_{\psi(B_2)}f-\av_{\tilde B_2}f\Big|+\Big|\av_{\psi(B_1)}f-\av_{\tilde B_1}f\Big| &\lesssim& \|f\|_{\ba}\big(1+V(t)\big)\big(|\ln r_2|^{-\alpha}+|\ln r_1|^{-\alpha}\big)\\
&\lesssim & \|f\|_{\ba}\big(1+V(t)\big)|\ln r_t|^{-\alpha}\\
&\lesssim& \|f\|_{\ba}
\end{eqnarray*}
which gives in turn
\begin{eqnarray*}
\frac{\mathcal{J}_1}{ 1+F(\frac{\ln r_2}{\ln r_1})}& \lesssim& \|f\|_{\ba}.
\end{eqnarray*}
Since  \mbox{$\tilde B_2\subset 2\tilde B_1$}   and  \mbox{$r_{2\tilde B_1}\le \frac12$},  then 
$$
\mathcal{J}_2\lesssim F\left(\frac{\ln r_{2,\psi}}{\ln r_{1,\psi}}\right) \|f\|_{\ba}.
$$
Hence we get from the property $(2)$ of the Definition \ref{def657} combined with \eqref{eqss:1} 
\begin{align*}
{\frac{\mathcal{J}_2  }{1+ F(\frac{\ln r_2}{\ln r_1} )}}& \leq \frac{F\left(\frac{\ln r_{2,\psi}}{\ln r_{1,\psi}}\right)}{1+F(\frac{\ln r_2}{\ln r_1} )}\|f\|_{\ba} \\
 &  \lesssim  \left(1+F\Big(\frac{\ln r_{2,\psi}}{\ln r_2}\frac{\ln r_1}{\ln r_{1,\psi}}\Big)\right)\|f\|_{\ba}\\
 &\lesssim  \|f\|_{\ba}.
\end{align*}
Since $\tilde{B_1}$ and $2\tilde{B_1}$ are comparable and $r_{2\tilde B_1}\le \frac12$ we easily have
$$
\frac{\mathcal{J}_3}{1+F(\frac{\ln r_2}{\ln r_1} )}\lesssim \mathcal{J}_3\lesssim  \|f\|_{{\ba}}.
$$  
The proof of Theorem \ref{decom} is now achieved.
\end{proof}
\subsection{Application to Euler equations}
In this section we shall deal with the local and  global well-posedness theory for the two dimensional Euler equations  in the space $\ba$.  This project will be performed through the use of the logarithmic estimate developed in Theorem \ref{decom}. We shall now state a more  general result than Theorem \ref{main}. \begin{Theo}
 \label{apriori} Let  \mbox{$\omega_0\in \ba\cap L^p$} with  $\alpha\in(0,1)$ and $p\in(1,2)$. Then, 
 \begin{enumerate}
 \item If  $F$ belongs to the class $\mathcal{A}$, there exists $T>0$ such that the system \eqref{E}  admits a unique local solution 
 $$
 \omega \in L^\infty([0,T]; {\it L^\alpha mo}_{1+F}).
 $$
 \item If $F$ belongs to the class $\mathcal{A}^\prime$, the    system \eqref{E}  admits a unique global solution 
 $$\omega \in L^\infty_{\textnormal{loc}}(\RR_+; {\it L^\alpha mo}_{1+F}).
 $$

 \end{enumerate}
 \end{Theo}
 \begin{proof} 
  The proof is based on the establishment of the {\it  a priori estimates} which are the cornerstone for  the existence and the  uniqueness parts. Here we omit the details about the existence and the uniqueness  which are classical and some of their elements  
  can be found for example in the paper \cite{BK2}.  
  
 $ {\bf (1)}$ Using Theorem \ref{decom} one has
 \begin{equation}\label{formal1}
  \|\omega(t)\|_{\baa\cap L^p}\lesssim \|\omega_0\|_{\ba \cap L^p}\left(1+V(t) F\big(2+V^{\frac{1}{\alpha}}(t) \big)\right)
 \end{equation}
  with $ \displaystyle{V(t)=\int_0^t\|u(\tau)\|_{L^{1-\alpha} L}d\tau.}$ Combining this estimate with Proposition \ref{coro} implies after integration in time
  \begin{equation}\label{linea1}
  V(t)\lesssim\|\omega_0\|_{\ba\cap L^p}\left(t+\int_0^t V(\tau) F(2+V^{\frac1\alpha}(\tau))d\tau\right).
  \end{equation}
  According to the Remark \ref{rmq23} the function $F$ has at most a polynomial growth: $F(2+x)\lesssim 1+x^{\beta}$. Therefore
  $$
    V(t)\lesssim\|\omega_0\|_{\ba\cap L^p}\Big(t+ t\, V(t)+t \,V^{1+\frac{\beta}{\alpha}}(t) \Big)
  $$
  and consequently we can find $T\triangleq T(\|\omega_0\|_{\ba\cap L^p})>0$  such that 
  $$
  \forall t\,\in [0,T],\quad V(t)\le 1.
  $$
  Plugging this estimate into \eqref{formal1} gives
  $$
  \|\omega(t)\|_{\baa\cap L^p}\lesssim  \|\omega_0\|_{\ba\cap L^p}. 
  $$
  Now from Theorem \ref{decom} we get also
  $$
   \|\omega(t)\|_{{\it L^\alpha mo}_{1+F}\cap L^p}\lesssim  \|\omega_0\|_{\ba\cap L^p}. 
  $$
  ${\bf{(2)}}$
Fix $T>0$ an arbitrary number, then from \eqref{linea1} we deduce
$$
\forall\, t\in [0,T],\quad  V(t)\le C\|\omega_0\|_{\ba\cap L^p}T+C\|\omega_0\|_{\ba \cap L^p} \Big( \int_0^t V(\tau) F(2+V^{\frac1\alpha}(\tau))d\tau \Big)
$$
and introduce the function $\mathcal{M}:[a,+\infty[\to [0,+\infty[$ defined by
$$
\mathcal{M}(y)=\int_{a}^{y}\frac{1}{x\, F(2+x^{\frac1\alpha})}dx,\quad a=\inf(A_T,1)\,\quad \, A_T=C\|\omega_0\|_{\ba\cap L^p}\, T.
$$
Since $F$ belongs to the class $\mathcal{A}^\prime$ we can easily check that
$$
\int_{A_T}^{+\infty}\frac{1}{x\, F(2+x^{\frac1\alpha})}dx=+\infty.
$$
Therefore applying Lemma \ref{osgood1}  
$$
 \forall t\in[0,T],\quad V(t)\le \mathcal{M}^{-1}\Big( \mathcal{M}(A_T)+C\|f\|_{\ba} t \Big).
$$
This gives the global a priori estimates
$$
\forall t\geq 0,\quad V(t)\le \mathcal{M}^{-1}\Big( \mathcal{M}(A_t)+C\|f\|_{\ba} t\Big).
$$
Inserting this estimate into \eqref{formal1} allows to get a global estimate for vorticity. Hence, there exists a continuous function  $G:\R_+\to \R_+$ related to $\mathcal{M}$ such that
\begin{equation}\label{basse313}
\|\omega(t)\|_{\baa\cap L^p}\le G(t).
\end{equation}
According to Theorem \ref{decom} and the preceding estimate
$$
\|\omega(t)\|_{\it{L^\alpha mo}_{1+F}}\le G(t).
$$
This concludes the a priori estimates.
\end{proof}

\appendix{}
\section{Technical lemmata}
We will prove the following lemma used before in the inequality \eqref{precise}.

\begin{Lemm} 
\label{equivalence}
There exists a universal implicit constant  such that for $r\in(0,r_t]$ and for $k\in \mathbb{N}$,
$$
u_k\lesssim  \delta r^2 e^{-k} e^{\delta (V(t)+1)(k-\ln(r))^{1-\alpha}}.
$$
\end{Lemm}
\begin{proof}
If we denote by   \mbox{$c_0 \geq 1$}  the implicit constant appearing in Whitney Lemma, then 
$$
u_k\leq \Big|\Big\{ y\in \psi(B)\ \backslash \ d(y, \psi(B)^c)\leq c_0e^{-k}r\Big\}\Big|.
$$
The preservation of  Lebesgue  measure by  \mbox{$\psi$}  yields
$$
 \Big|\Big\{ y\in \psi(B) \ \backslash \ d(y, \psi(B)^c)\leq c_0 e^{-k}r\big\}\big|=\Big|\Big\{ x\in B \ \backslash \ d(\psi(x), \psi(B)^c)\leq c_0 e^{-k}r\Big\}\Big|.
$$
Since  \mbox{$  \psi(B)^c=\psi(B^c)$}  then
$$
u_k\leq \Big|\Big\{ x\in B \ \backslash \ d(\psi(x), \psi(B^c))\leq c_0 e^{-k}r\Big\}\Big|.
$$
We set
  $$D_k=\Big\{ x\in B \ \backslash \ d(\psi(x), \psi(B^c))\leq c_0 e^{-k}r\Big\}.
  $$
Since  \mbox{$\psi(\partial B)$}  is the frontier of  \mbox{$\psi(B)$}  and  \mbox{$d(\psi(x), \psi(B^c))=d(\psi(x), \partial \psi(B))$}  then
$$
D_k\subset \Big\{ x\in B \ \backslash \ \exists y\in \partial B \;{\rm with}\; |\psi(x)- \psi(y)|\leq c_0 e^{-k}r\Big\}.
$$
The regularity of $\psi^{-1}$ (Proposition \ref{prop}) implies since $c_0 e^{-k}r\leq c_0r_t \lesssim \ell(t)$
$$
D_k\subset \Big\{ x\in B \ \backslash \  \exists y\in \partial B: |x- y|\le c_0 e^{-k}r e^{ \delta(V(t)+1)(k-\ln(r))^{1-\alpha}}\Big\}.
$$
Here we choose $\delta$ large enough such that $\delta>\ln(c_0)$ and $\delta>c_0$.
Thus,  \mbox{$D_k$}  is contained in the annulus
$$\mathcal A=\Big\{ x\in B \ \backslash \  d(x,\partial B) \le \delta e^{-k}r e^{ \delta (V(t)+1)(k-\ln(r))^{1-\alpha}}\Big\}$$  and so (since we are in dimension $2$)
$$
u_k\leq |D_k|\lesssim  \delta r^2 e^{-k} e^{\delta (V(t)+1)(k-\ln(r))^{1-\alpha}},$$
as claimed.
\end{proof}
We conclude this paper by the following result which has been used in several places.
\begin{Lemm}\label{maj12}
Let $\alpha\in]0,1[, A,B,C,D>0$ and $F: [1,+\infty[\to\R_+$ be a differentiable nondecreasing function such that
$$
C>D\quad\hbox{and}\quad \|F^\prime\|_{L^\infty}\triangleq M<+\infty.
$$ 
Consider the sequence
$$
w_n\triangleq\sum_{k\geq n}e^{-k+A(k+B)^{1-\alpha}} F\left(\frac{k+C}{D}\right).
$$
Assume that
$$\frac{A(1-\alpha)}{(n+B)^\alpha}\leq \frac14,
$$
 then
$$
w_n\le 4 e^{-n+A(n+B)^{1-\alpha}} F(\frac{n+C}{D})+ \frac{16M}{D} e^{-n+A(n+B)^{1-\alpha}}. 
$$
\end{Lemm}
\begin{proof}
Let $R_n:=\sum_{k\geq n}e^{-k}$ and $v_n:=e^{A(n+B)^{1-\alpha}} F(\frac{n+C}{D})$. According to Abel's formula
$$
w_n = R_n v_n + \sum_{k\geq n+1}R_k(v_k-v_{k-1}).
$$
It is clear that $0<R_n \le\frac{e^{-n}}{1-e^{-1}}$ and
$$
w_n\le\frac{1}{{1-e^{-1}}} e^{-n+A(n+B)^{1-\alpha}} F(\frac{n+C}{D})+\frac{1}{1-e^{-1}}\sum_{k\geq n+1}e^{-k}|v_k-v_{k-1}|.
$$
It remains to estimate the last sum. For this purpose we use the mean value theorem combined with the nondecreasing property of $F$
\begin{eqnarray*}
|v_k-v_{k-1}|\le \frac{A(1-\alpha)}{(k-1+B)^\alpha} e^{A(k+B)^{1-\alpha}} F(\frac{k+C}{D})+\frac{M}{D}e^{A(k+B)^{1-\alpha}}.
\end{eqnarray*}
Consequently
\begin{eqnarray*}
\sum_{k\geq n+1}R_k|v_k-v_{k-1}|&\le& \frac{A(1-\alpha)}{(1-e^{-1})(n+B)^\alpha} w_{n+1}+\frac{M}{(1-e^{-1})D}\sum_{k\geq n+1} e^{-k+A(k+B)^{1-\alpha}}\\
&\le&\frac{A(1-\alpha)}{(1-e^{-1})(n+B)^\alpha} w_{n}+\frac{M}{(1-e^{-1})D}\sum_{k\geq n} e^{-k+A(k+B)^{1-\alpha}}.
\end{eqnarray*}
This leads to 
$$
w_n\le \frac{1}{{1-e^{-1}}} e^{-n+A(n+B)^{1-\alpha}} F(\frac{n+C}{D})+\frac{A(1-\alpha)}{(1-e^{-1})(n+B)^\alpha} w_n+\frac{M}{(1-e^{-1})D}\sum_{k\geq n} e^{-k+A(k+B)^{1-\alpha}}. 
$$
Reproducing the same computations with $F(x)=1, M=0,$ we get
$$
\sum_{k\geq n} e^{-k+A(k+B)^{1-\alpha}}\le  \frac{1}{{1-e^{-1}}} e^{-n+A(n+B)^{1-\alpha}} +\frac{A(1-\alpha)}{(1-e^{-1})(n+B)^\alpha} \sum_{k\geq n} e^{-k+A(k+B)^{1-\alpha}}.$$
Assuming that
$$
\frac{A(1-\alpha)}{(1-e^{-1})(n+B)^\alpha}\leq \frac12,
$$
we get
$$
\sum_{k\geq n} e^{-k+A(k+B)^{1-\alpha}}\le  \frac{2}{{1-e^{-1}}} e^{-n+A(n+B)^{1-\alpha}} 
$$
and
$$
w_n\le \frac{2}{{1-e^{-1}}} e^{-n+A(n+B)^{1-\alpha}} F(\frac{n+C}{D})+ \frac{4M}{(1-e^{-1})^2D} e^{-n+A(n+B)^{1-\alpha}}. 
$$
Since $\frac{1}{1-e^{-1}}\le 2$ we deduce
$$
w_n\le 4 e^{-n+A(n+B)^{1-\alpha}} F(\frac{n+C}{D})+ \frac{16M}{D} e^{-n+A(n+B)^{1-\alpha}} .
$$
\end{proof}

\end{document}